\documentclass[12pt]{amsart}
\usepackage{amssymb}
\usepackage{amsthm}
\usepackage{mathrsfs}
\usepackage{amssymb}
\usepackage{amsfonts}
\usepackage[english]{babel}
\usepackage[T1]{fontenc}
\usepackage[latin1]{inputenc}
\usepackage{fullpage}
\usepackage{amsmath}
\usepackage[all]{xy}
\usepackage{stmaryrd}

\newtheorem{theorem}{Theorem}[section]
\newtheorem{lemma}[theorem]{Lemma}

\theoremstyle{definition}
\newtheorem{definition}[theorem]{Definition}

\newtheorem{proposition}[theorem]{Proposition}
\newtheorem{conjecture}[theorem]{Conjecture}

\theoremstyle{remark}
\newtheorem{remark}[theorem]{Remark}
\def\hh{\overset{\circ}{h}}
\def\dvg{dv_{\bar{g}}}
\newcommand{\bg}{\bar{g}}
\newcommand{\funchg}[1]{\mathcal{H}_{\bar{g}}#1}
\newcommand{\funch}{\mathcal{H}_{\bar{g}}}

\numberwithin{equation}{section}

\makeatletter
\newsavebox{\@brx}
\newcommand{\llangle}[1][]{\savebox{\@brx}{\(\m@th{#1\langle}\)}%
  \mathopen{\copy\@brx\kern-0.5\wd\@brx\usebox{\@brx}}}
\newcommand{\rrangle}[1][]{\savebox{\@brx}{\(\m@th{#1\rangle}\)}%
  \mathclose{\copy\@brx\kern-0.5\wd\@brx\usebox{\@brx}}}
\makeatother

\begin{document}

\title{Volume comparison theorem with respect to sigma-2 curvature}

\author{Jiaqi Chen}
\address{School of Electrical Engineering and Automation, Xiamen University of Technology, Xiamen 361024, P.R. China}
\email{chenjiaqi@xmut.edu.cn}

\author{Yi Fang}
\address{Department of Mathematics, Anhui University of Technology, Ma'anshan 243002, P.R.China}
\email{flxy85@163.com}

\author{Yan He}
\address{Faculty of Mathematics and Statistics, Hubei Key Laboratory of Applied Mathematics, Hubei University,  Wuhan 430062, P.R. China}
\email{helenaig@hotmail.com}

\author{Jingyang Zhong}
\address{College of Mathematics and Computer Sciences,  Fuzhou University, Fuzhou 350108, P.R. China}
\email{jyzhong@fzu.edu.cn}



\thanks{The first author was supported by the Scientific Research Foundation of Xiamen University of Technology Under Grant No. YKJ23009R. The second author was supported by National Natural Science Foundation of China(11801006, 12071489). The third author was supported by funds from Hubei Provincial Department of Education Key Projects D20171004. }


\begin{abstract}
  In this paper, we investigate the volume comparison theorem related to $\sigma_2$-curvature. In particular,
  we show that volume comparison theorem with respect to $\sigma_2$-curvature holds for metrics close to strictly stable positive Einstein metrics. By applying similar techniques, we derive the local rigidity theorem for strictly stable Ricci flat manifolds with respect to $\sigma_2$-curvature, which shows it admits no metric with positive $\sigma_2$-curvature near strictly stable Ricci-flat metrics.
\end{abstract}

\maketitle



\section{Introduction}\label{sec:introduction}
Let $(M^n, g)$ be a closed Riemannian manifold of dimension $n$. The Ricci curvature tensor and scalar curvature are denoted by $\text{Ric}_g$ and $\text{R}_g$ respectively. The classic volume comparison theorem states that if
\begin{equation*}
    \text{Ric}_g \geq (n-1)g,
\end{equation*}
then
\begin{equation*}
\text{V}_{M^n}(g) \leq \text{V}_{\mathbb{S}^n}(g_{\mathbb{S}^n})
\end{equation*}
where $\mathbb{S}^n$ is the unit round sphere and $g_{\mathbb{S}^n}$ is the canonical metric. Motivated by this volume comparison theorem, R. Schoen \cite{S} proposed the following conjecture on volume comparison theorems involving scalar curvature on closed hyperbolic manifolds.

\begin{conjecture}
    Suppose $(M^n, \bg)$ be any closed hyperbolic manifold of dimension $n \geq 3$. Then for any metric $g$ on $M^n$ with
    \begin{equation*}
        R_g \geq R_{\bg},
    \end{equation*}
    then the following volume comparison
    \begin{equation*}
        \text{V}_{M^n}(g) \geq \text{V}_{M^n}(\bg)
    \end{equation*}
    holds.
\end{conjecture}

Schoen's conjecture was proved confirmatively in 3-dimensional case due to the works of Hamilton \cite{H} and Perelman \cite{P}. For higher dimensional cases, Besson, Courtois and Gallot \cite{BCG} verified it for metrics $C^2$-close to the canonical metric $\bg$. After that, they also showed that the same result holds without assuming the metrics $C^2$-closeness if one replaces the assumption on scalar curvature by Ricci curvatures.

H.Bray \cite{B} posed the following volume comparison conjecture related to Ricci curvature tensor and scalar curvature, and he gave affirmative proof for 3-dimensional case. And the conjecture is still open for $n \geq 4$.

\begin{conjecture}
    There exists a constant $\varepsilon_n \in (0,1)$ such that for any closed Riemannian manifold $(M^n,g)$ satisfying
    \begin{align*}
      \text{R}_g \geq n(n-1), \ \ \  \text{Ric}_g \geq \varepsilon_n(n-1)g,
    \end{align*}
    then the volume comparison
    \begin{align*}
      \text{V}_{M^n}(g) \leq \text{V}_{\mathbb{S}^n}(g_{\mathbb{S}^n})
    \end{align*}
    holds, where $\mathbb{S}^n$ is the unit round sphere and $g_{\mathbb{S}^n}$ is the canonical metric.
\end{conjecture}

If $g$ is sufficiently closed to the stable Einstein metric $\bg$, i.e. $||g-\bg||_{C^2(M^2,\bg)} < \varepsilon_0$ for some small constant $\varepsilon_0 > 0$, then Yuan \cite{Y} proved the volume comparison theorem with respect to scalar curvature holds. And for the volume comparison of asymptotically hyperbolic manifold, please refer to \cite{HJS}.

Before we state our results, we give the definition of the stable Einstein metric.

\begin{definition}
    For $n \geq 3$, suppose $(M^n, \bg)$ is a closed Einstein manifold with
    \begin{align*}
    \text{Ric}_{\bg} = (n-1) \lambda \bg.
    \end{align*}
    The Einstein metric $\bg$ is said to be strictly stable, if the Einstein operator $\Delta_E^{\bg} = \Delta_{\bg}+ 2Rm_{\bg}$ is a negative operator on $S^{TT}_{2, \bg}(M) \setminus \{0\}$, where
    \begin{align*}
    S^{TT}_{2, \bg}(M):=\{ h\in S_2(M)| \delta_{\bg} h=0, tr_{\bg}h=0\}.
    \end{align*}
    is the space of transverse-traceless symmetric $2$-tensor on $(M^n, \bg)$. Moreover, $\bg$ is called strictly stable is the Einstein operator $\Delta_{\bg}$ is a strictly negative operator.
\end{definition}

There are several different ways to define the stability of Einstein metric, and we choose the one with Einstein operator for our convenience. With this definition, Yuan proved the following volume comparison theorem with respect to scalar curvature.

\begin{theorem}(\cite{Y})
    Suppose $(M^n,\bar{g})$ is a strictly stable Einstein manifold with
    \begin{align*}
      \text{Ric}_{\bar{g}}=(n-1)\lambda\bar{g}.
    \end{align*}
    Then there exists a constant $\varepsilon_0>0$ for any metric $g$ on $M^n$ satisfying
    \begin{equation*}
    R_g\geq R_{\bar{g}}, \ \ \ ||g-\bar{g}||_{C^2(M^2,\bar{g})}<\varepsilon_0,
    \end{equation*}
    the following volume comparison hold:
    \begin{enumerate}
      \item if $\lambda>0$, then $\text{Vol}_{M}(g)\leq \text{Vol}_M(\bar{g})$;
      \item if $\lambda<0$, then $\text{Vol}_{M}(g)\geq \text{Vol}_M(\bar{g})$,
    \end{enumerate}
    with equality holds in either case if and only if $g$ is isometric to $\bar{g}$.
\end{theorem}

Instead of scalar curvature, we investigate the $\sigma_2$-curvature tensor, a fourth order curvature tensor which has been intensively studied for decades in dimension four and higher dimensions. For a closed 4-dimensional Riemannian manifold $(M^4, g)$, $\sigma_2$-curvature is defined to be:

\begin{equation*}
  \sigma_2 = - \frac{1}{2} |Ric_g|^2 + \frac{1}{6}R_g^2.
\end{equation*}
It satisfies the Gauss-Bonnet-Chern formula:
\begin{equation*}
  \int_{M^4} \bigg( \frac{1}{4}|W_g|^2 + \sigma_2\bigg)dv_g = 8\pi^2 \chi(M),
\end{equation*}
where $W_g$ is the Weyl curvature tensor for $(M^4, g)$. In particular, if $(M^4, g)$ is locally conformally flat, i.e. $W_g = 0$, it reduces to
\begin{equation*}
  \int_{M^4} \sigma_2 dv_g = 8\pi^2 \chi(M).
\end{equation*}
Also, by the conformal invariance of Weyl curvature tensor, $\sigma_2$ is also conformal invariant in four dimensional case.

In this article, we proved the following volume comparison theorem with respect to $\sigma_2$:
\begin{theorem}\label{theorem: main1}
  Let $(M^n,\bar{g})$ be a strictly stable Einstein manifold with
  \begin{align*}
    Ric_{\bar{g}}=(n-1)\lambda\bar{g}
  \end{align*}
  where $\lambda>0$ is a constant. Then there exists a constant $\varepsilon_0>0$ for any metric $g$ on $M^n$ satisfying
  \begin{align*}
    \sigma_2(g)\geq \sigma_2(\bar{g}), \ \ \ ||g-\bar{g}||_{C^2(M^2,\bar{g})}<\varepsilon_0,
  \end{align*}
  the following volume comparison holds:
  \begin{align*}
    \text{Vol}_{M}(g)\leq \text{Vol}_M(\bar{g}),
  \end{align*}
   with equality holds in either case if and only if $g$ is isometric to $\bar{g}$.
\end{theorem}

\begin{remark}
  The strictly stable condition is necessary, otherwise we could construct counterexamples.
\end{remark}

\begin{remark}
  Our theorem is slightly different from Yuan's results. In our theorem, the Einstein manifolds must have positive constant scalar curvature.
\end{remark}

\begin{remark}
  In dimension $n = 4$, we have a global volume comparison which is a direct consequence from Gauss-Bonnet-Chern formula. Suppose that $(M^4, \bg)$ is a closed locally conformally flat Riemannian manifold with positive constant $\sigma_2$-curvature, then for any metric $g$ satisfies $\sigma_2(g) \geq \sigma_2(\bg)$ pointwisely on $M$, we have
  \begin{align*}
    Vol_M(g) \leq Vol_M(\bg) - \frac{1}{4\sigma_2(\bg)} || W_g ||_{L^2(M, g)}^2 \leq Vol_M(\bg)
  \end{align*}
  where equality holds if and only if $g$ is also locally conformally flat.
\end{remark}

\begin{remark}
  The volume comparison for Ricci-flat manifolds can not be expected. This is easy to see by taking $g = c^2 \bar{g}$ for some non zero constant $c$. Clearly, in this case, $\sigma_2(g) = \sigma_2(\bar{g}) = 0$, but the volume $\text{Vol}_{M}(g)$ can be either larger or smaller than $\text{Vol}_M(\bar{g})$ depending on $c$ is either greater or smaller than 1.
\end{remark}

However in this case we hold the rigidity as follows.

\begin{theorem}\label{theorem: main2}
  Suppose $(M^n,\bar{g})$ is a strictly stable Ricci-flat manifold, then there exists a constant $\varepsilon_0>0$ for any metric $g$ on $M^n$ satisfying
  \begin{align*}
    \sigma_2(g)\geq 0, \ \ \ ||g-\bar{g}||_{C^2(M^2,\bar{g})}<\varepsilon_0,
  \end{align*}
  then $g$ is Ricci-flat.
\end{theorem}

\begin{remark}
  Since $\sigma_2$-curvature vanishes on Ricci-flat metric, Theorem \ref{theorem: main2} states that there is no metric with positive $\sigma_2$-curvature near $\bar{g}$.
\end{remark}

\begin{remark}
  When we finished our work, we found the paper \cite{ACS} also cover the volume comparison theorem with respect to $\sigma_2$-curvature with a different approach.
\end{remark}

This paper is organized as follows. In Sect.\ref{sec:preliminary}, we give some notations and preliminaries that will be useful throughout this paper. In Sect.\ref{sec:functional}, we will calculate the first and second variation of the given functional. In Sect.\ref{sec:main}, we will prove the main theorem and construct a counterexample to show the necessity of the stable Einstein condition.

\vskip 0.1 in
{\bf Acknowledgements.} The authors are grateful to Professor Wei Yuan for useful discussions and comments on this work.

\section{Preliminary}\label{sec:preliminary}

\subsection{Notations}

Throughout this paper, we will always assume that $(M^n, g)$ to be an n dimensional closed Riemannian manifold with $(n \geq 3)$ unless otherwise stated. Also, we list the curvature indices convention we used in this article.

Given a Riemannian manifold $(M^n,g)$, We adopt the following convention for Riemann and Ricci curvature tensors
\begin{align*}
  Rm(X,Y,Z,W) = \langle R(X, Y)Z, W \rangle,
\end{align*}
where
\begin{align*}
  R(X, Y)Z = \nabla_X \nabla_Y Z - \nabla_Y \nabla_X Z - \nabla_{[X, Y]}Z.
\end{align*}
Thus we define
\begin{align*}
  R_{ijkl} &= Rm(\partial_i, \partial_j, \partial_k, \partial_l) = g_{ml}R_{ijk}^m\\
  R_{jk} &= g^{il}R_{ijkl},
\end{align*}
and denote the Schouten tensor as
\begin{align*}
  S_{ij}=R_{ij}-\frac{R}{2(n-1)}g_{ij}.
\end{align*}
Then the $\sigma_k$-curvatures are defined as the $k$th elementary symmetric polynomials
\begin{equation*}
  \sigma_k(\lambda_1,\cdots,\lambda_n)=\sum\limits_{i_1<i_2<\cdots<i_k}\lambda_{i_1}\lambda_{i_2}\cdots\lambda_{i_k},
\end{equation*}
where $\lambda_i(i=1,\cdots,n)$ are the eigenvalues of the Schouten tensor $S_g$. Easy to see that
\begin{equation*}
  \sigma_1(g)=tr(S_g)=\frac{n-2}{2(n-1)}R,
\end{equation*}
thus $\sigma_k$-curvatures are often considered as a canonical way to generalize the scalar curvature. By simple calculations, we get
\begin{equation*}
  \sigma_2(g)=-\frac{1}{2}|Ric_g|^2+\frac{n}{8(n-1)}R^2_g.
\end{equation*}
We also denote the Weyl tensor as
\begin{align*}
  W_{ijkl}=R_{ijkl}-\frac{1}{n-2}(S_{il}g_{jk}+S_{jk}g_{il}-S_{ik}g_{jl}-S_{jl}g_{ik}).
\end{align*}

In the case that $(M^n,\bar{g})$ is an Einstein manifold that satisfies $Ric_{\bar{g}}=(n-1)\lambda \bar{g}$, we hold the following:
\begin{align*}
  S_{ij}=\frac{n-2}{2}\lambda \bar{g}_{ij}, \quad\quad W_{ijkl}=R_{ijkl}-\lambda(\bar{g}_{il}\bar{g}_{jk}-\bar{g}_{ik}\bar{g}_{jl}),\\
  R_{\bar{g}}=n(n-1)\lambda,\quad\quad  \sigma_2(\bar{g})=\frac{n(n-1)(n-2)^2}{8}\lambda^2.
\end{align*}

For Laplacian operator, we use the following convention:
\begin{align*}
  \Delta_g = g^{ij}\nabla_i \nabla_j.
\end{align*}

For simplicity, we introduce following operators:
\begin{align*}
  (h\times k)_{ij}=g^{kl}h_{ik}h_{jl},\quad \quad h \cdot k=tr(h\times k)=g^{ij}g^{kl}h_{ik}k_{jl},
\end{align*}
and
\begin{align*}
  (Rm \cdot h)_{jk}=R_{ijkl}h^{il},\quad \quad (W \cdot h)_{jk}=W_{ijkl}h^{il}.
\end{align*}
Then the Einstein operator acting on symmetric 2-tensor $h \in S_2(M)$ is defined by
\begin{align*}
  \Delta_E h = \Delta h + 2 Rm \cdot h.
\end{align*}

\subsection{Basic variational formulae}
We list several formulae for linearization of curvature tensor which will show later sections.

The linearization of Ricci tensor is
\begin{align*}
  (DRic)\cdot h=-\frac{1}{2}[\Delta_E h-Ric\times h-h\times Ric+\nabla^2(trh)+\nabla_j(\delta h)_k+\nabla_k(\delta h)_j],
\end{align*}
  and the linearization of scalar curvature is
\begin{align*}
  (DR)\cdot h=-\Delta(trh)+\delta^2h-Ric\cdot h.
\end{align*}

The first and second variations of the volume functional are
\begin{align*}
  (D\text{Vol})\cdot h=\frac{1}{2}\int_M(trh)dv_g,
\end{align*}
and
\begin{align*}
  (D^2\text{Vol})(h,h)=\frac{1}{4}\int_M\Big((trh)^2-2|h|^2\Big)dv_g.
\end{align*}

\section{The Key Functional and Corresponding Variations}\label{sec:functional}

In this section, we will present the variation formulae for functional $\mathcal{H}_{\bar{g}}$ defined below. Throughout this section, $\bar{g}$ will be an Einstein metric, that is
\begin{equation*}
  Ric_{\bar{g}}=(n-1) \lambda \bar{g}.
\end{equation*}

Suppose $g$ is another metric near $\bar{g}$, then by theorem there is a diffeomorphism $\phi\in\mathcal{D}(M)$ such that $h:=\phi^*g-\bar{g}$ satisfies the TT-gauge condition
\begin{equation*}
\delta_{\bar{g}}(\hh)=\delta_{\bar{g}}(h-\frac{trh}{n}g)=0.
\end{equation*}
Later we denote $trh$ as $u$, it is clear that we have the following
\begin{equation*}
(\delta h)_k=-\nabla^j h_{jk}=-\frac{1}{n}\nabla_k u, \quad |h|^2=|\hh|^2+\frac{1}{n}u^2, \quad |\nabla h|^2=|\nabla \hh|^2+\frac{1}{n}|\nabla u|^2.
\end{equation*}

Suppose $(M^n, \bar{g})$ is a closed Einstein manifold, take the functional
\begin{equation*}
\mathcal{H}_{\bar{g}}(g)=\text{Vol}(g)^{\frac{4}{n}}\int_M\sigma_2(g)dv_{\bar{g}}.
\end{equation*}
Since the volume form $dv_{\bar{g}}$ is independent of $g$, it is easy to check that $\mathcal{H}_{\bar{g}}(g)$ is a scaling invariant with respect to metric $g$, that is
\begin{align*}
  \mathcal{H}_{\bar{g}}(c^2 g) = c^2\mathcal{H}_{\bar{g}}(g)
\end{align*}
for any non-zero constant $c$. Such volume form fixed functional has been studied by Fisher-Marsden \cite{FM} and Yuan \cite{Y}\cite{LY2}, where the functional was used to study rigidity theorem of scalar curvature and Q curvature and corresponding volume comparison theorems. The functional $\mathcal{H}_{\bar{g}}(g)$ we introduce here is designed for our proof of volume comparison and we will present its first and second variation formulae in this section.

\begin{proposition}\label{prop:einstein_critical}
  The Einstein metric $\bar{g}$ is the critical point of $\mathcal{H}_{\bar{g}}$.
\end{proposition}

\begin{proof}
  Let $g_t=\bar{g}+th$, where $h\in S_2(M)$, $t\in(-\varepsilon,\varepsilon)$ for some small $\varepsilon>0$. Then the linearization of the functional $\mathcal{H}_{\bar{g}}$ at $\bar{g}$ is
  \begin{equation*}
    \begin{split}
      D\mathcal{H}_{\bar{g}}\cdot h
      =&\frac{d}{dt}\bigg|_{t=0}\mathcal{H}_{\bar{g}}(g_t)\\
      =&\text{Vol}(\bar{g})^{\frac{4}{n}}\int_MD\sigma_2(\bar{g})\cdot h dv_{\bar{g}}
      +\frac{4}{n}\text{Vol}(\bar{g})^{\frac{4}{n}-1}D\text{Vol}(\bar{g})\cdot h\int_M\sigma_2(\bar{g})dv_{\bar{g}}.
    \end{split}
  \end{equation*}
  Due to the linearization of Ricci tensor and the scalar curvature, we obtain the following:
  \begin{align*}
    (D\sigma_2)\cdot h&=-\frac{1}{2}(|Ric|^2)'+\frac{n}{8(n-1)}(R^2)'\\
    &=-(g^{ik})'g^{jk}R_{ij}R_{kl}-g^{ik}g^{jl}(R'_{ij})R_{kl}+\frac{n}{4(n-1)}RR'.
  \end{align*}
  Notice that in the case $Ric=(n-1)\lambda g$, we obtain $R=n(n-1)\lambda$ and
  \begin{align*}
    g^{ij}R_{ij}'=-\frac{n-1}{n}\Delta u,\quad\quad R'=(1-n)(\frac{1}{n}\Delta u+\lambda u),
  \end{align*}
  thus
  \begin{align*}
    (D\sigma_2)\cdot h=
    &(n-1)^2\lambda^2 u-\frac{(n-1)^2}{n}\lambda\Delta u+\frac{n^2(1-n)\lambda}{4}(\frac{1}{n}\Delta u+\lambda u)\\
    =&-\frac{(n-1)(n-2)^2}{4}\lambda^2 u+c(n)\lambda \Delta u,
  \end{align*}
  which implies
  \begin{align*}
    \int_M (D\sigma_2)\cdot h\dvg=-\frac{(n-1)(n-2)^2}{4}\lambda^2\int_M u \dvg,
  \end{align*}
  and
  \begin{align*}
    D\mathcal{H}_{\bar{g}}\cdot h
    &=-\frac{(n-1)(n-2)^2}{4}\lambda^2\text{Vol}(\bar{g})^{\frac{4}{n}}\int_M u \dvg\\
    &+\frac{4}{n}\text{Vol}(\bar{g})^{\frac{4}{n}-1}\cdot \frac{n(n-1)(n-2)^2}{8}\lambda^2\text{Vol}(\bar{g})\cdot\frac{1}{2}\int_M u\dvg\\
    &=0.
  \end{align*}
  Thus $\bar{g}$ is the critical point of $\funchg(g)$.
\end{proof}

In the following, we will calculate the second order variation formula of the functional $\mathcal{H}_{\bar{g}}$. With a standard method, we have
\begin{equation*}
  \begin{split}
    D^2\mathcal{H}_{\bar{g}}(h,h)
    =&\text{Vol}(\bar{g})^{\frac{4}{n}}\int_MD^2\sigma_2(\bar{g})(h,h)dv_{\bar{g}}+\frac{8}{n}\text{Vol}(\bar{g})^{\frac{4}{n}-1}(D\text{Vol}(\bar{g})\cdot h)\int_MD\sigma_2(\bar{g})\cdot hdv_{\bar{g}}\\
    &+\frac{4}{n}\Big[\text{Vol}(\bar{g})^{\frac{4}{n}-1}D^2\text{Vol}(\bar{g})(h,h)+\frac{4-n}{n}\text{Vol}(\bar{g})^{\frac{4}{n}-2}(D\text{Vol}(\bar{g})\cdot h )^2\Big]\int_M\sigma_2(\bar{g})dv_{\bar{g}}.
  \end{split}
\end{equation*}
Note that since $\bar{g}$ is assumed to be Einstein, we have the following results that can help us minimize the work. The first and second variation formulae reduce to:
\begin{align*}
  D\text{Vol}(\bar{g})\cdot h=\frac{1}{2}\int_M u dv_{\bar{g}},
\end{align*}
and
\begin{align*}
  D^2\text{Vol}(\bar{g})(h,h)&=\frac{1}{4}\int_M u^2-2|h|^2dv_{\bar{g}}=\frac{n-2}{4n}\int_M u^2 dv_{\bar{g}}-\frac{1}{2}\int_M |\hh|^2dv_{\bar{g}}.
\end{align*}
The integrals involved with $\sigma_2$ and its first variation will be
\begin{align*}
  \int_M\sigma_2(\bar{g})dv_{\bar{g}}&=\frac{n(n-1)(n-2)^2}{8}\text{Vol}(\bar{g}),
\end{align*}
and
\begin{align*}
    \int_M D\sigma_2(\bar{g})\cdot hdv_{\bar{g}}&=\int_M\Lambda_{\bar{g}}(h)dv_{\bar{g}}=\int_M \Lambda^*(1)\cdot h dv_{\bar{g}}=-\frac{(n-1)(n-2)^2}{4}\int_M udv_{\bar{g}}
\end{align*}
And one can check that
\begin{align*}
  (\int_M u dv_{\bar{g}})^2=\text{Vol}(\bar{g})\int_M \bar{u}^2dv_{\bar{g}},
\end{align*}
here
\begin{align*}
  \bar{u}=\text{Vol}(\bar{g})^{-1}\int_M u dv_{\bar{g}},
\end{align*}
thus we have:
\begin{equation*}
  \begin{split}
    \text{Vol}(\bar{g})^{-\frac{4}{n}}D^2\mathcal{H}_{\bar{g}}(h,h)=&\int_MD^2\sigma_2(\bar{g})(h,h)dv_{\bar{g}}-\frac{(n-1)(n-2)^2}{4}\int_M |\hh|^2dv_{\bar{g}}\\
    &+\frac{(n-1)(n-2)^3}{8n}\int_M u^2dv_{\bar{g}}-\frac{(n-1)(n-2)^2(n+4)}{8n}\int_M \bar{u}^2dv_{\bar{g}}.
  \end{split}
\end{equation*}

Now focus on the calculation of $D^2\sigma_2(\bar{g})(h,h)$, note that in our case $R_{ij}=(n-1)\lambda g_{ij}$.
\begin{equation*}
  \begin{split}
    D^2\sigma_2(\bar{g})(h,h)=&-\frac{1}{2}(|Ric|^2)''+\frac{n}{8(n-1)}(R^2)''\\
    =&-\frac{1}{2}(g^{ik}g^{jl}R_{ij}R_{kl})''+\frac{n}{4(n-1)}(R\cdot R')'\\
    =&-((g^{ik})'g^{jl}R_{ij}R_{kl})'-(g^{ik}g^{jl}(R_{ij})'R_{kl})'+\frac{n}{4(n-1)}(R')^2+\frac{n^2}{4}R''\\
    =&-(n-1)^2\lambda^2\bar{g}_{ik}(g^{ik})''-(n-1)^2\lambda^2\bar{g}_{ij}\bar{g}_{kl}(g^{ik})'(g^{jl})'-4(n-1)\lambda(g^{ik})'(R_{ik})'\\
    &-\bar{g}^{ik}\bar{g}^{jl}(R_{ij})'(R_{kl})'-(n-1)\lambda\bar{g}^{ij}R''_{ij}+\frac{n}{4(n-1)}(R')^2+\frac{n^2}{4}\lambda R''\\
    =&-3(n-1)^2\lambda^2|h|^2+4(n-1)\lambda h^{ij}R'_{ij}-|Ric'|^2-(n-1)\lambda \bar{g}^{ij}R''_{ij}\\
    &+\frac{n}{4(n-1)}(R')^2+\frac{n^2}{4}\lambda R''
  \end{split}
\end{equation*}
Notice that
\begin{align*}
  R''=(g^{ij})''R_{ij}+2(g^{ij})'R'_{ij}+\bar{g}^{ij}R''_{ij}=2(n-1)\lambda|h|^2-2h^{ij}R'_{ij}+\bar{g}^{ij}R''_{ij},
\end{align*}
thus
\begin{equation*}
  \begin{split}
    D^2\sigma_2(\bar{g})(h,h)=&\frac{(n-1)(n^2-6n+6)}{2}\lambda^2|h|^2+\frac{-n^2+8n-8}{2}\lambda h^{ij}R'_{ij}-|Ric'|^2\\
    &+\frac{(n-2)^2}{4}\lambda \bar{g}^{ij}R''_{ij}+\frac{n}{4(n-1)}(R')^2
  \end{split}
\end{equation*}

With above calculations, we can express the integrals in a compact way. We split the result in following porpositions.

\begin{proposition}\label{prop:3.2}
  \begin{equation*}
    \int_M h^{ij}R'_{ij} dv_{\bar{g}}=-\frac{1}{2}\int_M \hh \cdot \Delta_L \hh dv_{\bar{g}}+\frac{n-1}{n^2}\int_M|\nabla u|^2 dv_{\bar{g}}.
  \end{equation*}
\end{proposition}

\begin{proof}
  First, we denote
  \begin{equation*}
    \Delta_L h_{jk}:=\Delta h_{jk}+2R_{ijkl}h^{il}-R_{jp}h^p_k-R_{pk}h^p_j=\Delta_E h_{jk}-R_{jp}h^p_k-R_{pk}h^p_j
  \end{equation*}
  $\Delta_L$ is the Lichnerowicz Laplacian. Recall that in our case $R_{ij}=(n-1)\lambda \bar{g}_{ij}$, thus $R_{ijkl}=W_{ijkl}+\lambda(\bar{g}_{il}\bar{g}_{jk}-\bar{g}_{ik}\bar{g}_{jl})$, we have
  \begin{equation*}
    \Delta_L h_{jk}=\Delta h_{jk}+2W_{ijkl}h^{il}+2\lambda(tr_{\bar{g}}h)\bar{g}_{jk}-2n\lambda h_{jk}.
  \end{equation*}
  Thus
  \begin{equation*}
    \begin{split}
      \Delta_L \hh_{jk}&=\Delta \hh_{jk}+2W_{ijkl}\hh^{il}-2n\lambda\hh_{jk}\\
      \Delta_L h_{jk}&=\Delta \hh_{jk}+\frac{\Delta u}{n}\bar{g}_{jk}+2W_{ijkl}\hh^{il}-2n\lambda\hh_{jk}=\Delta_L \hh_{jk}+\frac{\Delta u}{n}\bar{g}_{jk}.
    \end{split}
  \end{equation*}
  From linearization of Ricci tensor we have
  \begin{equation*}
    \begin{split}
      R'_{jk}&=-\frac{1}{2}(\Delta_L h_{jk}+\nabla_j\nabla_k trh+\nabla_j(\delta h)_k+\nabla_k(\delta h)_j)\\
      &=-\frac{1}{2}(\Delta_L \hh_{jk}+\frac{\Delta u}{n}\bar{g}_{jk}+\frac{n-2}{n}\nabla_j\nabla_k u)
    \end{split}
  \end{equation*}
  Thus
  \begin{equation*}
    \begin{split}
      \int_M h^{jk}R'_{jk} dv_{\bar{g}}=&-\frac{1}{2}\int_M (\hh^{jk}+\frac{u}{n}g^{jk})(\Delta_L \hh_{jk}+\frac{\Delta u}{n}\bar{g}_{jk}+\frac{n-2}{n}\nabla_j\nabla_k u) dv_{\bar{g}}\\
      =&-\frac{1}{2}\int_M \hh \cdot \Delta_L \hh dv_{\bar{g}}-\frac{1}{2n}\int_M u\Delta u dv_{\bar{g}}-\frac{n-2}{2n^2}\int_M u\Delta u dv_{\bar{g}}\\
      =&-\frac{1}{2}\int_M \hh \cdot \Delta_L \hh dv_{\bar{g}}+\frac{n-1}{n^2}\int_M|\nabla u|^2 dv_{\bar{g}}
    \end{split}
  \end{equation*}
\end{proof}

\begin{proposition}\label{prop:3.3}
  \begin{equation*}
    \int_M |Ric'|^2\dvg=\frac{1}{4}\int_M \Delta_L^2 \hh\cdot \hh\dvg+\frac{n-1}{4n}\int_M (\Delta u)^2\dvg-\frac{(n-1)(n-2)^2}{4n^2}\lambda\int_M |\nabla u|^2\dvg.
  \end{equation*}
\end{proposition}

\begin{proof}
  First we show that $\int_M \Delta_L\hh\cdot \nabla^2 u\dvg=0$.
  \begin{equation*}
    \begin{split}
      \int_M \Delta_L\hh\cdot \nabla^2 u\dvg=&\int_M (\Delta\hh+2W\cdot \hh-2n\lambda\hh)\cdot\nabla^2 u\dvg\\
      =&\int_M \Delta\hh\cdot \nabla^2 u\dvg+2\int_M W(\hh,\nabla^2 u)\dvg.
    \end{split}
  \end{equation*}
  here we use the fact that $\delta\hh=0$, on the other hand
  \begin{equation*}
    \begin{split}
      \int_M \Delta\hh\cdot \nabla^2 u\dvg=&\int_M \nabla_k\nabla^k \hh^{ij}\nabla_i\nabla_j u\dvg=-\int_M \nabla_i \nabla_k\nabla^k \hh^{ij} \nabla_j u\dvg\\
      =&-\int_M (\nabla_k\nabla_i\nabla^k \hh^{ij}+R_{ikp}^k \nabla^p\hh^{ij}+R_{ikp}^i \nabla^k\hh^{pj}+R_{ikp}^j \nabla^k\hh^{ip} )\nabla_j u\dvg\\
      =&\int_M \nabla_i\nabla_k \hh^i_j\nabla^k\nabla^j u+\int_M R_{ikjp}\nabla^k\hh^{ip}\nabla^j u\dvg\\
      =&\int_M (\nabla_k\nabla_i\hh^i_j+R_{ikp}^i\hh^p_{j}-R_{ikj}^p\hh_p^i)\nabla^k\nabla^j u-\int_M R_{ikjp}\hh^{ip}\nabla^k\nabla^j u\dvg\\
      =&-2\int_M R_{ikjp}\hh^{ip}\nabla^k\nabla^j u\dvg=-2\int_M W(\hh,\nabla^2 u).
    \end{split}
  \end{equation*}
  here we use the fact $\nabla^k R_{ikjp}=0$ from the second Bianchi identity.Now we have
  \begin{equation*}
    \begin{split}
      \int_M |Ric'|^2\dvg=&\frac{1}{4}\int_M|\Delta_L\hh+\frac{\Delta u}{n}\bar{g}+\frac{n-2}{n}\nabla^2 u|^2\dvg\\
      =&\frac{1}{4}\int_M|\Delta_L \hh|^2+\frac{1}{n}(\Delta u)^2+\frac{(n-2)^2}{n^2}|\nabla^2 u|^2+\frac{2(n-2)}{n^2}(\Delta u)^2\dvg\\
      =&\frac{1}{4}\int_M \Delta_L^2 \hh\cdot \hh+\frac{3n-4}{n^2}(\Delta u)^2+\frac{(n-2)^2}{n^2}|\nabla^2 u|^2\dvg.
    \end{split}
  \end{equation*}
  From Bochner formula we obtain
  \begin{equation*}
    \int_M |\nabla^2 u|^2\dvg=\int_M (\Delta u)^2\dvg-(n-1)\lambda\int_M |\nabla u|^2 \dvg.
  \end{equation*}
  Thus
  \begin{align*}
    \int_M |Ric'|^2\dvg = &\frac{1}{4}\int_M \Delta_L^2 \hh \cdot \hh\dvg+\frac{n-1}{4n}\int_M (\Delta u)^2\dvg \\
    &-\frac{(n-1)(n-2)^2}{4n^2}\lambda\int_M |\nabla u|^2\dvg.
  \end{align*}
\end{proof}

\begin{proposition}\label{prop:3.4}
  \begin{equation*}
    \int_M g^{ij}R''_{ij}\dvg=-\frac{1}{2}\int_M \Delta_L\hh\cdot\hh\dvg-\frac{(n-1)(n-2)}{2n^2}\int_M |\nabla u|^2\dvg.
  \end{equation*}
\end{proposition}

\begin{proof}
  From \cite{Y} we obtain
  \begin{equation*}
    \begin{split}
      R_{jk}''=&h^{pi}(R_{ijkl}h^l_p+R_{ijpl}h^l_k-\nabla_i\nabla_kh_{jp}+\nabla_i\nabla_ph_{jk}+\nabla_j\nabla_kh_{ip}-\nabla_j\nabla_ph_{ik})\\
      &+\frac{1}{2}(\nabla_jh^p_k+\nabla_kh^p_j-\nabla^ph_{jk})(2(\delta h)_p+\nabla_p u)\\
      &+\frac{1}{2}(\nabla_ih^p_k+\nabla_kh^p_i-\nabla^ph_{ik})(\nabla^ih_{jp}+\nabla_jh^i_p-\nabla_ph^i_j).
    \end{split}
  \end{equation*}
  Recall that $(\delta h)_p=-\frac{1}{n}\nabla_p u$ and $R_{ij}=(n-1)\lambda g_{ij}$, we obtain
  \begin{equation*}
    \begin{split}
      g^{jk}R''_{jk}=&(n-1)\lambda |h|^2-R(h,h)+\frac{n-1}{n}h\cdot \nabla^2 u+h\cdot \Delta h-h^{pi}\nabla_j\nabla_p h^j_i\\
      &-\frac{(n-2)^2}{2n^2}|\nabla u|^2+\frac{3}{2}|\nabla h|^2-\nabla_i h_{jk}\nabla^kh^{ij}\\
      =&n\lambda |\hh|^2-W(\hh,\hh)+\frac{n-1}{n}\hh\cdot \nabla^2u+\hh\cdot\Delta\hh+\frac{2n-1}{n^2}u\Delta u-\frac{(n-2)^2}{2n^2}|\nabla u|^2\\
      &+\frac{3}{2}|\nabla \hh|^2+\frac{3}{2n}|\nabla u|^2-h^{pi}\nabla_j\nabla_p h^j_i-\nabla_i h_{jk}\nabla^kh^{ij}.
    \end{split}
  \end{equation*}
  Thus
  \begin{equation*}
    \begin{split}
      \int_M g^{ij}R''_{ij}\dvg=&\int_M\Big( n\lambda |\hh|^2-W(\hh,\hh)\Big)\dvg+\frac{1}{2}\int_M|\nabla\hh|^2\dvg\\
      &-\frac{n^2-3n+2}{2n^2}\int_M |\nabla u|^2\dvg - \int_M h^{pi}\nabla_j\nabla_p h^j_i\dvg-\int_M \nabla_i h_{jk}\nabla^kh^{ij}\dvg.\\
      =&-\frac{1}{2}\int_M \Delta_L\hh\cdot\hh\dvg-\frac{(n-1)(n-2)}{2n^2}\int_M |\nabla u|^2\dvg\\
      &-\int_M \nabla_i (h_{jk}\nabla^k h^{ij}) \dvg\\
      =&-\frac{1}{2}\int_M \Delta_L\hh\cdot\hh\dvg-\frac{(n-1)(n-2)}{2n^2}\int_M |\nabla u|^2\dvg
    \end{split}
  \end{equation*}
\end{proof}

\begin{proposition}\label{prop:3.5}
  \begin{equation*}
  \int_M (R')^2\dvg=\frac{(n-1)^2}{n^2}\int_M (\Delta u)^2\dvg-\frac{2(n-1)^2}{n}\lambda\int_M|\nabla u|^2+(n-1)^2\lambda^2\int_M u^2\dvg.
  \end{equation*}
\end{proposition}

\begin{proof}
  In the case that $Ric=(n-1)\lambda \bar{g}$ we obtain that
  \begin{align*}
    R'=\delta^2 h-\Delta u-(n-1)\lambda u=(1-n)(\frac{1}{n}\Delta u+\lambda u),
  \end{align*}
  thus
  \begin{equation*}
    \begin{split}
      \int_M (R')^2\dvg=&(n-1)^2\int_M(\frac{1}{n}\Delta u+\lambda u)^2\dvg\\
      =&\frac{(n-1)^2}{n^2}\int_M (\Delta u)^2\dvg-\frac{2(n-1)^2}{n}\lambda\int_M|\nabla u|^2+(n-1)^2\lambda^2\int_M u^2\dvg.
    \end{split}
  \end{equation*}
\end{proof}

Now combine Proposition \ref{prop:3.2} to Proposition \ref{prop:3.5} we obtain
\begin{equation*}
  \begin{split}
    &\int_M D^2\sigma_2(\bar{g})(h,h)\dvg\\
    =&-\frac{1}{4}\int_M \Big(\Delta^2_L \hh\cdot\hh-\frac{n^2-12n+12}{2}\lambda\Delta_L \hh\cdot \hh\Big)\dvg+\frac{3(n-1)(n-2)^2}{4n}\lambda^2\int_M u^2\dvg\\
    &+\frac{(n-1)(n^2-6n+6)}{2}\lambda^2\int_M|\hh|^2\dvg-\frac{(n-1)(n-2)^2(n+4)}{8n^2}\lambda\int_M |\nabla u|^2\dvg.
  \end{split}
\end{equation*}

Finally we obtain
\begin{equation*}
  \begin{split}
    &\text{Vol}(\bar{g})^{-\frac{4}{n}}D^2\mathcal{H}_{\bar{g}}(h,h)\\
    =&-\frac{1}{4}\int_M \Big(\Delta^2_L \hh\cdot\hh-\frac{n^2-12n+12}{2}\lambda\Delta_L \hh\cdot \hh-(n-1)(n^2-8n+8)\lambda^2\hh\cdot\hh\Big)\dvg\\
    &+\frac{(n-1)(n-2)^2(n+4)}{8n}\Big(\lambda^2\int_M(u-\bar{u})^2\dvg -\frac{\lambda}{n}\int_M|\nabla u|^2\dvg \Big)\\
    =&-\frac{1}{4}\int_M \Big(\Delta_L+2(n-1)\lambda\Big)\Big(\Delta_L-\frac{n^2-8n+8}{2}\lambda\Big)\hh\cdot\hh\dvg\\
    &-\frac{(n-1)(n-2)^2(n+4)}{8n^2}\lambda\int_M\Big(|\nabla u|^2-n\lambda(u-\bar{u})^2\Big)\dvg
  \end{split}
\end{equation*}
Notice that in our case $\Delta_E=\Delta_L+2(n-1)\lambda$, the second variation formula of $\mathcal{H}_{\bar{g}}$ can be expressed as below:

\begin{proposition}\label{prop:second_order_var}
  Suppose $(M^n,\bar{g})$ is an Einstein manifold with
  \begin{align*}
    Ric_{\bar{g}}=\lambda (n-1)\bar{g},
  \end{align*}
  then
\begin{equation*}
  \begin{split}
    \text{Vol}(\bar{g})^{-\frac{4}{n}}D^2\mathcal{H}_{\bar{g}}(h,h)=&-\frac{1}{4}\int_M \Delta_E\Big(\Delta_E-\frac{(n-2)^2}{2}\lambda\Big)\hh\cdot\hh\dvg\\
    &-\frac{(n-1)(n-2)^2(n+4)}{8n^2}\lambda\int_M\Big(|\nabla u|^2-n\lambda(u-\bar{u})^2\Big)\dvg
  \end{split}
\end{equation*}
for any $h=\hh+\frac{1}{n}u\bar{g}\in S^{TT}_{2,\bar{g}}(M)\oplus  (C^\infty(M)\cdot\bar{g}).$
\end{proposition}

\section{Volume comparison with respect to $\sigma_2-$curvature}\label{sec:main}
In this section, we give the proof of our main result, before that we state some lemma we could use later.

A key step in our proof is to investigate the second order variation of the functional $\funch(g)$ at $\bar{g}$. According to the formula we derived in Proposition \ref{prop:second_order_var}, we need the following well-known Lichnerowicz-Obata's eigenvalue estimate \cite{LI}\cite{O}, we refer more details in Theorem 5.1 in \cite{L}.
\begin{lemma}[Lichnerowicz-Obata's eigenvalue estimate]\label{lemma:Lichnerowicz-Obata}
  Suppose $(M^n,\bar{g})$ is an $n$ dimensional closed Riemannian manifold with
  \begin{align*}
    Ric_{\bar{g}}\geq(n-1)\lambda \bar{g},
  \end{align*}
  where $ \lambda>0$ is a constant. Then for any function $f\in C^\infty(M)\setminus\{0\}$ with
  \begin{align*}
    \int_M f dv_{\bar{g}}=0,
  \end{align*}
  we have
  \begin{align*}
    \int_M|\nabla f|^2dv_{\bar{g}}\geq n\lambda\int_M f^2dv_{\bar{g}},
  \end{align*}
  where equality holds if and only id $(M^n,\bar{g})$ is isometric to the round sphere $\mathbb{S}^n(r)$ with radius $r=\frac{1}{\sqrt{\lambda}}$ and $f$ is a first eigenfunction of the Laplace-Beltrami operator.
\end{lemma}

Applying this to Proposition \ref{prop:second_order_var}, immediately we obtain the non-positive definite property of the second order variation of $\funch(g)$ and Einstein metric $\bg$.

\begin{proposition}\label{prop: characteristic_second_order}
  Suppose $(M^n,\bar{g})$ is a strictly stable Einstein manifold with Ricci curvature tensor
  \begin{align*}
    Ric_{\bar{g}}=(n-1)\lambda\bar{g},
  \end{align*}
  where $\lambda\geq 0$ is a constant, then $\bar{g}$ is a critical point of $\mathcal{H}_{\bar{g}}$ and
  \begin{align*}
    D^2\mathcal{H}_{\bar{g}}(h,h)\leq 0
  \end{align*}
  for any $h=\hh+\frac{1}{n}(tr_{\bar{g}}h)\bar{g}\in S^{TT}_{2,\bar{g}}(M)\oplus(C^{\infty}(M)\cdot \bar{g})$. Furthermore the equality holds if and only if
  \begin{itemize}
    \item $h\in \mathbb{R}\bar{g}$, when $(M^n ,\bar{g})$ is not isometric to the round sphere up to a rescaling of the metric.
    \item $h\in(\mathbb{R}\oplus E_{n\lambda})\bar{g}$, when $(M^n,\bar{g})$ is isometric to the round sphere $\mathbb{S}^n(r)$ with radius $\frac{1}{\sqrt{\lambda}}$,
    where
    \begin{align*}
      E_{n\lambda}=\{u\in C^{\infty}(\mathbb{S}^n(r))|\Delta_{\mathbb{S}^n(r)}u+n\lambda u=0\}
    \end{align*}
    is the space of first eigenfunctions for the spherical metric.
  \end{itemize}
\end{proposition}

\begin{proof}
  According to Proposition \ref{prop:einstein_critical}, $\bar{g}$ is a critical metric of $\mathcal{H}_{\bar{g}}$, from Proposition \ref{prop:second_order_var} we have
  \begin{equation*}
    \begin{split}
      \text{Vol}(\bar{g})^{-\frac{4}{n}}D^2\mathcal{H}_{\bar{g}}(h,h)=&-\frac{1}{4}\int_M \Delta_E\Big(\Delta_E-\frac{(n-2)^2}{2}\lambda\Big)\hh\cdot\hh\dvg\\
      &-\frac{(n-1)(n-2)^2(n+4)}{8n^2}\lambda\int_M\Big(|\nabla u|^2-n\lambda(u-\bar{u})^2\Big)\dvg
    \end{split}
  \end{equation*}
  We show that $D^2 \mathcal{H}_{\bar{g}}$ is non-positive definite according to Lemma \ref{lemma:Lichnerowicz-Obata} and $\Delta_E$ is positive definite operator. Furthermore, when $D^2 \mathcal{H}_{\bar{g}}\leq 0$, we have
  \begin{align*}
    \int_M \Delta_E\Big(\Delta_E-\frac{(n-2)^2}{2}\lambda\Big)\hh\cdot\hh\dvg=0
  \end{align*}
  and
  \begin{align*}
    \int_M\Big(|\nabla u|^2-n\lambda(u-\bar{u})^2\Big)\dvg=0
  \end{align*}
  The first equation implies that $\hh=0$ since $\Delta_E$ is positive definite, now apply Lemma \ref{lemma:Lichnerowicz-Obata} and let $f=u-\bar{u}$, there are two cases:
  \begin{itemize}
    \item $(M^n,\bar{g})$ is not isometric to the round sphere, in this case Lemma \ref{lemma:Lichnerowicz-Obata} implies that $u=\bar{u}$, in other words $u$ is a constant, thus $h=\hh+\frac{1}{n}u \bar{g}\in \mathbb{R}\bar{g}$.
    \item $(M^n,\bar{g})$ is isometric to the round sphere, in this case we hold $f\in E_{n\lambda}$, which implies $u=f+\bar{u} \in \mathbb{R}\oplus E_{n\lambda}$, thus $h=\hh+\frac{1}{n}u \bar{g}\in (\mathbb{R}\oplus E_{n\lambda})\bar{g}$.
  \end{itemize}
\end{proof}

Now we need to state Ebin-Palais slice theorem \cite{E}, this theorem suggests that one can define a local slice $\mathcal{S}_{\bar{g}}$ which is a set of equivalent classes of metrics near $\bar{g}$ modulo diffeomorphisms.

\begin{theorem}[Ebin-Palais slice theorem]\label{theorem: Ebin-Palais}
  Suppose $(M^n,\bar{g})$ is a closed $n-$dimensional Einstein manifold with Ricci curvature tensor
  \begin{align*}
    Ric_{\bar{g}}=(n-1)\lambda\bar{g}
  \end{align*}
  where $\lambda$ is a constant. Then there exists a local slice $\mathcal{S}_{\bar{g}}$ though $\bar{g}$. That is, one can find constant $\epsilon>0$ such that for any metric $g$ satisfies $||g-\bar{g}||_{C^2(M,\bar{g})}$, there is a diffeomorphism $\varphi$ and $\varphi^*g\in\mathcal{S}_{\bar{g}}$
  In particular, for a smooth local slice $\mathcal{S}_{\bar{g}}$, we obtain the decomposition
  \begin{align*}
    S_2(M)=T_{\bar{g}}\mathcal{S}_{\bar{g}}\oplus (T_{\bar{g}}\mathcal{S}_{\bar{g}})^\perp
  \end{align*}
  In the case that $(M^n,\bar{g})$ is not isometric to round sphere we have:
  \begin{align*}
    T_{\bar{g}}\mathcal{S}_{\bar{g}} &= S^{TT}_{2,\bar{g}}(M)\oplus (C^\infty(M)\cdot \bar{g})\\
    (T_{\bar{g}}\mathcal{S}_{\bar{g}})^\perp &= \{ \mathcal{L}_{\bar{g}}(X): \langle X, \nabla_{\bar{g}} u \rangle_{L^2(M, \bar{g})} \text{ for all } u \in C^\infty(M) \}
  \end{align*}
  and in the case that $(M^n,\bar{g})$ is isometric to the round sphere $\mathbb{S}^n(r)$ with radius $r = \frac{1}{\sqrt{\lambda}}$, we have:
  \begin{align*}
    T_{\bar{g}}\mathcal{S}_{\bar{g}} &=S^{TT}_{2,\bar{g}}(M)\oplus (E_{n\lambda}^\perp\cdot \bar{g})\\
    (T_{\bar{g}}\mathcal{S}_{\bar{g}})^\perp &= \{ \mathcal{L}_{\bar{g}}(X): \langle X, \nabla_{\bar{g}} u \rangle_{L^2(M, \bar{g})} \text{ for all } u \in E_{n\lambda}^\perp \}
  \end{align*}
  here,
  \begin{align*}
    E_{n\lambda}^\perp = \{ u \in C^\infty(\mathbb{S}^n(r)): \Delta_{\mathbb{S}^n(r)} u + n\lambda u = 0 \}
  \end{align*}
  is the space of first eigenfunctions for the spherical metric.
\end{theorem}

Now we restrict the functional $\mathcal{H}_{\bar{g}}(g)$ on a local slice $\mathcal{S}_{\bar{g}}$, denoted by $\mathcal{H}_{\bar{g}}|_{\mathcal{S}_{\bar{g}}}(g)$, before we proceeding to the main proof, we need the following Morse lemma on Banach manifolds which was proved by Fisher and Marsden \cite{FM}.

\begin{lemma}[Morse lemma]\label{lemma: morse}
  Let $\mathcal{P}$ be a Banach manifold and $F: \mathcal{P} \rightarrow \mathbb{R}$ a $C^2-$function. Suppose that $\mathcal{Q} \subset \mathcal{P}$ is a submanifold, $F = 0$ and $dF = 0$ on $\mathcal{Q}$ and that there is a smooth normal bundle neighborhood of $\mathcal{Q}$ such that if $\mathcal{E}_x$ is the normal complement to $T_x\mathcal{Q}$ in $T_x\mathcal{P}$, then $d^2F(x)$ is weakly negative definite on $\mathcal{E}_x$(i.e. $d^2F(x)(v,v) \leq 0$ with equality only if $v = 0$). Let $\langle \cdot, \cdot \rangle_x$ be a weak Riemannian structure with a smooth connnection and assume that $F$ has a smooth $\langle \cdot, \cdot \rangle_x$-gradient, $Y(x)$. Assume that $DY(x): \mathcal{E}_x \rightarrow \mathcal{E}_x$ is an isomorphism for $x \in \mathcal{Q}$. Then there is a neighborhood $U$ of $\mathcal{Q}$ such that $y \in U$ and $F(y) \geq 0$ implies that $y \in \mathcal{Q}$.
\end{lemma}

Combine Proposition \ref{prop: characteristic_second_order} and Theorem \ref{theorem: Ebin-Palais}, we can find a local slice $\mathcal{S}_{\bar{g}}$ through $\bar{g}$, and we consider $\mathcal{Q}_{\bar{g}}$ as a submanifold of $\mathcal{S}_{\bar{g}}$ where
\begin{align*}
  \mathcal{Q}_{\bar{g}} := \{ c^2 \bar{g} \in \mathcal{S}_{\bar{g}} | c \neq 0 \}
\end{align*}
Applying previous lemma, we obtain the following regidity result:

\begin{proposition}\label{prop:rigidity}
  Suppose $(M^n,\bar{g})$ is a strictly stable Einstein manifold with Ricci curvature
  \begin{align*}
    Ric_{\bar{g}}=(n-1)\lambda\bar{g}
  \end{align*}
  where $\lambda>0$ is a constant. There is a neighbourhood of $\bar{g}$ in the local slice $\mathcal{S}_{\bar{g}}$, denoted by $U_{\bar{g}}$, such that any metric $g\in U_{\bar{g}} $ satisfying
  \begin{align*}
    \mathcal{H}_{\bar{g}}|_{\mathcal{S}_{\bar{g}}}(g)\geq\mathcal{H}_{\bar{g}}|_{\mathcal{S}_{\bar{g}}}(\bar{g})
  \end{align*}
  implies that $g=c^2\bar{g}$ for some positive constant $c$.
\end{proposition}

\begin{proof}
  From Proposition \ref{prop: characteristic_second_order}, we can see that $\bar{g}$ is the critical point of functional $\mathcal{H}_{\bar{g}}|_{\mathcal{S}_{\bar{g}}}$, and $D^2\mathcal{H}_{\bar{g}}|_{\mathcal{S}_{\bar{g}}}$ is non-positive definite on $T_{\bar{g}}\mathcal{S}_{\bar{g}}$. Furthermore, from Proposition \ref{prop:second_order_var}, we can see that $D^2\mathcal{H}_{\bar{g}}|_{\mathcal{S}_{\bar{g}}}$ is degenrate if and only if we restrict on
  \begin{align*}
    T_{\bar{g}}\mathcal{Q}_{\bar{g}} = \mathbb{R}\bar{g}
  \end{align*}
  From Theorem \ref{theorem: Ebin-Palais}, let $\mathcal{E}_{\bar{g}}$ be the $L^2$-orthogonal complement of $T_{\bar{g}}\mathcal{Q}$ in $T_{\bar{g}}\mathcal{S}_{\bar{g}}$, then base on the proposition,

  \begin{itemize}
    \item if $\bar{g}$ is not spherical, we have:
    \begin{align*}
      \mathcal{E}_{\bar{g}} = \bigg\{ h \in S^{TT}_{2,\bar{g}}(M) \oplus (C^\infty(M) \cdot \bar{g}) \bigg| \int_M tr_{\bar{g}} h dv_{\bar{g}} \bigg\}
    \end{align*}
    \item if $\bar{g}$ is spherical, we have:
    \begin{align*}
      \mathcal{E}_{\bar{g}} = \bigg\{ h \in S^{TT}_{2,\bar{g}}(M) \oplus (E_{n \lambda}^\perp \cdot \bar{g}) \bigg| \int_M tr_{\bar{g}} h dv_{\bar{g}} \bigg\}
    \end{align*}
  \end{itemize}
  Therefore, $D^2\mathcal{H}_{\bar{g}}|_{\mathcal{S}_{\bar{g}}}$ is strictly negative on $\mathcal{E}_{\bar{g}}$.

  As in \cite{Y}, we introduce the following weak Riemannian structure
  \begin{align*}
    \llangle h, k\rrangle_{g_s} := \int_M \langle h, k \rangle_{g_s} +  \langle \nabla_{g_s}h, \nabla_{g_s}k \rangle_{g_s} dv_{g_s} = \int_M \langle (1 - \Delta_{g_s}) h, k \rangle_{g_s} dv_{g_s}
  \end{align*}
  According to \cite{E}, this weak Riemannian structure has a smooth connnection and the $\llangle \cdot,\cdot \rrangle_{g_s}$ gradient of $\mathcal{H}_{\bar{g}}|_{\mathcal{S}_{\bar{g}}}$ is given by:
  \begin{align*}
    Y(g_s) = P_{g_s} (1 - \Delta_{g_s})^{-1} \bigg[ V_M(g_s)^{\frac{4}{n}}  \bigg( \Gamma^*_{g_s}(\rho_{g_s}) + \frac{2}{n} g_s V_M(g_s)^{-\frac{n+4}{n}}  \mathcal{H}_{\bar{g}}(g_s)\bigg)\bigg]
  \end{align*}
  where $P_{g_s}$ is the orthogonal projection to $T_{g_s}\mathcal{S}_{\bar{g}}$,  $\Gamma^*: C^\infty(M) \rightarrow S_2(M)$ is the $L^2-$adjoint of $D\sigma_2(g_s)$ and $\rho_{g_s} > 0$ is a smooth function on $M$ satisfying $dv_{\bar{g}} = \rho_{g_s} dv_{g_s}$. For our convience, we define the folowing vector field:
  \begin{align*}
    Z(g_s) :=  V_M(g_s)^{\frac{4}{n}}  \bigg( \Gamma^*_{g_s}(\rho_{g_s}) + \frac{2}{n} g_s V_M(g_s)^{-\frac{n+4}{n}}  \mathcal{H}_{\bar{g}}(g_s)\bigg).
  \end{align*}
  Since we assume that $\bar{g}$ is Einstein, $Z(\bar{g}) = 0$ and Furthermore,
  \begin{align*}
    (DZ_{\bar{g}}) \cdot h = D^2 \mathcal{H}_{\bar{g}}|_{\mathcal{S}_{\bar{g}}}(h, \cdot)
  \end{align*}
  for any $h \in \mathcal{E}_{\bar{g}}$, thus we have:
  \begin{align*}
    DY_{\bar{g}} = P_{g_s} (1 - \Delta_{g_s})^{-1}(DZ_{\bar{g}})
  \end{align*}
  is an isomorphism on $\mathcal{E}_{\bar{g}}$ because $D^2\mathcal{H}_{\bar{g}}|_{\mathcal{S}_{\bar{g}}}$ is strictly negative on $\mathcal{E}_{\bar{g}}$.

  According to proevious Morse lemma \ref{lemma: morse}, we can find a neighborhood $\mathcal{U}_{\bar{g}} \subset \mathcal{S}_{\bar{g}}$ such that for any metric $g_s \in \mathcal{U}_{\bar{g}}$ satisfying
  \begin{align*}
    \mathcal{H}_{\bar{g}}|_{\mathcal{S}_{\bar{g}}}(g_s) \geq \mathcal{H}_{\bar{g}}|_{\mathcal{S}_{\bar{g}}}(\bar{g})
  \end{align*}
  implies that $g_s \in \mathcal{Q}_{\bar{g}}$, which means that there is a non-zero constant $c$ such that $g_s = c^2 \bar{g}$
\end{proof}

Now we can prove the main theorem:
\begin{proof}[Proof of Theorem \ref{theorem: main1}]
  We can find a positive constant $\epsilon<\epsilon_0$ such that for any metric $g$ satisfies
  \begin{align*}
    ||g-\bar{g}||_{C^2(M,\bar{g})}<\epsilon
  \end{align*}
  there exists a diffeomorphism $\varphi$ such that $\varphi^*g\in \mathcal{U}_{\bar{g}}\subset \mathcal{S}_{\bar{g}}$, where $\mathcal{U}_{\bar{g}}$ is defined in Proposition \ref{prop:rigidity}.

  Notice that $D\mathcal{H}_{\bar{g}}\cdot h=0$ and $D^2\mathcal{H}_{\bar{g}}(h,h)\leq 0$, thus
  \begin{align*}
    \mathcal{H}_{\bar{g}}|_{\mathcal{S}_{\bar{g}}}(\varphi^*g)=\text{Vol}(\varphi^*g)^{\frac{4}{n}}\int_M \sigma_2(\varphi^*g)dv_{\bar{g}}\geq \text{Vol}(\bar{g})^{\frac{4}{n}}\int_M \sigma_2(\bar{g})dv_{\bar{g}}=\mathcal{H}_{\bar{g}}|_{\mathcal{S}_{\bar{g}}}(\bar{g}),
  \end{align*}

  we hold that $\varphi^*g=c^2g$ for some constant $c>0$ according to Proposition \ref{prop:rigidity}. Now assume that we have the reversed condition, which is $\text{Vol}(g)\geq\text{Vol}(\bar{g})$, it is obvious to see that
  \begin{align*}
    \text{Vol}(g)=c^n \text{Vol}(\bar{g})\geq \text{Vol}(\bar{g})
  \end{align*}
  which implies $c\geq 1$, but in other hands we hold
  \begin{align*}
    \sigma_2(g)=c^{-4}\sigma_2(\bar{g})\geq \sigma_2(\bar{g})
  \end{align*}
  which implies $c\leq 1$, thus $c=1$ which is to say $g$ is diffeomorphic to $\bar{g}$ and we finish the proof.
\end{proof}

When $\lambda=0$, we do not have the volume comparison, while we hold that in this case $g$ is conformal to $\bar{g}$, now we show the proof of Theorem \ref{theorem: main2}:

\begin{proof}[Proof of Theorem \ref{theorem: main2}]
    Recall Proposition 3.6 and let $\lambda=0$ we hold
  \begin{align*}
    \text{Vol}(\bar{g})^{-\frac{4}{n}}D^2\mathcal{H}_{\bar{g}}(h,h)=&-\frac{1}{4}\int_M |\Delta_E \hh|^2\dvg \leq 0
  \end{align*}
  Thus $D^2\mathcal{H}_{\bar{g}}$ is non-positive definite, since $\bar{g}$ is a critical metric for $\mathcal{H}_{\bar{g}}$, following the same method in the previous proof, there exists a diffeomorphism $\varphi$ such that $\varphi^*g \in U_{\bar{g}}\subset \mathcal{S}_{\bar{g}}$ and $\mathcal{H}_{\bar{g}}|_{\mathcal{S}_{\bar{g}}}(\varphi^*g)\geq\mathcal{H}_{\bar{g}}|_{\mathcal{S}_{\bar{g}}}(\bar{g})=0.$ This implies that $g=c^2\bar{g}$ thus $g$ is Ricci-flat.
\end{proof}

\begin{remark}
The stability condition is necessary, otherwise we have counterexamples. If the metric $\bg$ is unstable, according to Proposition \ref{prop:second_order_var}, $\bg$ is nolonger a local maximum. According to \cite{K}, a product of positive Einstein manifolds with identical Einstein constant is still Einstein but not stable, therefore, we can construct such product manifold as our counterexample. Let $(S^2,g_{S^2})$ be the unit sphere with canonical metric. Consider the canonical product manifold $(S^2\times S^2\times S^2\times S^2, \bar{g})$, where $\bar{g}=g_{S^2}+g_{S^2}+g_{S^2}+g_{S^2}$. Then we have
\begin{align*}
  \sigma(\bar{g})=-\frac{1}{2}|Ric|^2_{\bar{g}}+\frac{1}{7}R^2_{\bar{g}}=\frac{36}{7}.
\end{align*}
If we choose
\begin{align*}
  g_t=(1+4t^2)^{-1}g_{S^2}+(1+4t^2)^{-1}g_{S^2}+(1+3t)^{-1}g_{S^2}+(1-3t)^{-1}g_{S^2}
\end{align*}
then we have
\begin{equation*}
  \begin{split}
    \sigma(g_t)
    =&-\frac{1}{2}[2(1+4t^2)^2+2(1+4t^2)^2+2(1+3t)^2+2(1-3t)^2]\\
    &+\frac{1}{7}[2(1+4t^2)+2(1+4t^2)+2(1+3t)+2(1-3t)]^2\\
    =&-[2(1+8t^2+16t^4)+2(1+9t^2)]+\frac{4}{7}(4+8t^2)^2\\
    =&-(4+34t^2+32t^4)+\frac{64}{7}(1+4t^2+4t^4)\\
    =&\frac{36}{7}+\frac{18}{7}t^2+\frac{32}{7}t^4.
  \end{split}
\end{equation*}
It follows that
\begin{align*}
  \sigma(g_t)\geq\sigma(\bar{g}),
\end{align*}
and the equality holds if and only if $t=0$.

On the other hand, for the volume, we have
\begin{equation*}
  \begin{split}
    \text{Vol}(g_t)&=(1+4t^2)^{-1}(1+4t^2)^{-1}(1+3t)^{-1}(1-3t)^{-1}\text{Vol}(\bar{g})\\
    &=\frac{1}{1-t^2-56t^4-144t^6}\text{Vol}(\bar{g})
  \end{split}
\end{equation*}
It follows that
\begin{align*}
  \text{Vol}(g_t)\geq \text{Vol}(\bar{g})
\end{align*}
if $|t|$ is chosen small enough and the equality holds if and only if $t=0$.

So if we choose $0\leq|t|\leq \varepsilon$  for some small number $\varepsilon>0$,  we have
\begin{equation*}
  \begin{split}
    \sigma(g_t)>\sigma(\bar{g}), \ \ \  \text{Vol}(g_t)>\text{Vol}(\bar{g}).
  \end{split}
\end{equation*}
\end{remark}

\end{document}